\documentclass{amsart}
\usepackage{amsmath}
\usepackage{amsfonts}
\usepackage{amssymb}
\usepackage{amsrefs}
\usepackage{amsthm}
\newtheorem*{WDF}{Weyl Dimension Formula}
\newtheorem*{thmn}{Theorem (Gross and Wallach)}
\newtheorem*{mthm}{Main Theorem}
\newtheorem*{thm}{Theorem}

\newtheorem*{lem}{Lemma}

\numberwithin{equation}{section}

\begin{document}
\title{A multi-variate generating function \\ for the Weyl Dimension Formula}
\author{Wayne Johnson}
\address{Department of Mathematics, University of Wisconsin--Milwaukee}
\email{waj@uwm.edu}

\begin{abstract}
We present a closed form for a multi-variate generating function for the dimensions of the irreducible representations of a semisimple, simply connected linear algebraic group over $\mathbb{C}$ whose highest weights lie in a finitely generated lattice cone in the dominant chamber. This result generalizes the formula for the Hilbert series of an equivariant embedding of a homogeneous projective variety. As a special case, we show how the multi-variate series can be used to compute the Hilbert series of the determinantal varieties.
\end{abstract}

\maketitle

\section{Introduction}

Let $G$ be a semisimple, simply connected linear algebraic group over $\mathbb{C}$, and fix a choice $T\subset B\subset G$ of maximal torus and Borel subgroup. The choice of Borel gives us a set of positive roots $\Phi^+$ for $\mathfrak{g}:=Lie(G)$, and a set $P_+(\mathfrak{g})$ of dominant integral weights for $\mathfrak{g}$. To each $\lambda\in P_+(\mathfrak{g})$, the Theorem of the Highest Weight gives us a finite dimensional irreducible representation $L(\lambda)$ of $G$. Using this representation, we can find a parabolic subgroup $P\supset B$--namely, $P$ is the subgroup of $G$ that stabilizes the unique hyperplane $H$ in $L(\lambda)$ fixed by $B$. Then we get an embedding of $G/P$ into the projective space of all hyperplanes in $L(\lambda)$, denoted $\mathbb{P}(L(\lambda))$, given by $\pi_{\lambda}(gP):=g(H)$.
\\
\indent In the paper ``On the Hilbert polynomials and Hilbert series of homogeneous projective varieties" \cite{MR2906911}, the authors present a closed form for the Hilbert series of the equivariant embedding $\pi_{\lambda}$ of $G/P$ into $\mathbb{P}(L(\lambda))$. In particular, they show that the homogeneous coordinate ring $A(G/P)$ of such an embedding is isomorphic to $\displaystyle\bigoplus_{n\in\mathbb{N}} L(n\lambda)$, where $L(\lambda)$ denotes the above irreducible representation of $G$ with highest weight $\lambda$. The Hilbert series of the embedding is then given by the function
\begin{center}
$HS_q(\lambda)=\displaystyle\sum_{n\in\mathbb{N}}dim(L(n\lambda))q^n$.
\end{center}
They then prove the Hilbert series has the following closed form.

\begin{thmn}
The Hilbert series of the embedding $\pi_{\lambda}$ of $G/P$ is
\begin{center}
$\displaystyle\prod_{\alpha\in\Phi^+}\left(\frac{(\lambda,\alpha)}{(\rho,\alpha)}q\frac{d}{dq}+1\right)\frac{1}{1-q}$.
\end{center}
\end{thmn}

A natural generalization of the above Hilbert series is the formal power series
\begin{equation}
HS_{\textbf{q}}\langle\lambda_1,\dots,\lambda_k\rangle:=\displaystyle\sum_{(a_1,\dots,a_k)\in\mathbb{N}^k}dim(L(a_1\lambda_1+\dots+a_k\lambda_k))q_1^{a_1}\dots q_k^{a_k},
\end{equation}
where $q_1,\dots,q_k$ are indeterminates, and $\lambda_1,\dots,\lambda_k$ are dominant integral weights. The main result of this paper is to prove a generalization of the above theorem and find a closed form of (1.1). We prove the following.

\begin{mthm}
Let $\lambda_1,\dots,\lambda_k$ be dominant integral weights. Then
\begin{equation}
HS_{\textbf{q}}\langle\lambda_1,\dots,\lambda_k\rangle=\displaystyle\prod_{\alpha\in\Phi^+}\left(1+c_{\lambda_1}(\alpha)q_1\frac{\partial}{\partial q_1}+\dots+c_{\lambda_k}(\alpha)q_k\frac{\partial}{\partial q_k}\right)\prod_{i=1}^k\frac{1}{1-q_i},
\end{equation}
where $c_{\lambda}(\alpha):=\displaystyle\frac{(\lambda,\alpha)}{(\rho,\alpha)}$.
\end{mthm}

This is a generating function for the dimensions of the finite dimensional irreducible representations of $G$ whose highest weight lies in the lattice cone in $P_+(\mathfrak{g})$ generated by $\lambda_1,\dots,\lambda_k$. We denote such a lattice cone by $\langle\lambda_1,\dots,\lambda_k\rangle$. Note that in \cite{MR1120029}, the authors give the special case where $\mathfrak{g}$ has rank $k$ and we choose $\lambda_i$ to be the fundamental dominant weight $\omega_i$ for $1\leq i\leq k$. The above theorem applies to a more general lattice cone.
\\
\indent Note that (1.2) allows us to compute the Hilbert series for many varieties by first computing the multi-variate series and then specializing to a gradation on the algebra
\begin{center}
$\displaystyle\bigoplus_{\lambda\in\langle\lambda_1,\dots,\lambda_k\rangle}L(\lambda)$,
\end{center}
via a suitable substitution. This is especially useful in computing the Hilbert series of determinantal varieties, which are traditionally quite difficult to compute (see, for example, \cite{MR2037715}). The series (1.2) is not difficult to compute using Mathematica or Maple, and then we find the Hilbert series by specializing the grade appropriately. For instance, in \S 4, we give a linear recursion on $n$ for computing (1.2) for the weights in $\langle2\omega_1,2\omega_2\rangle$, where $\omega_1$ and $\omega_2$ are the first two fundamental dominant weights of $SL(n,\mathbb{C})$, and then specialize this two variable series to obtain the Hilbert series of the determinantal variety of rank at most two symmetric matrices in $M_n(\mathbb{C})$. The methods presented in this paper bypass much of the complicated machinery traditionally used to compute these series.

\section{Preliminaries}

Throughout this paper, let $G$ be a semisimple, simply connected linear algebraic group over $\mathbb{C}$. Let $T$ be a maximal torus and $T\subset B\subset G$ a choice of Borel subgroup containing $T$. Let $U$ be the unipotent radical of $B$. We denote by $\mathfrak{g},\mathfrak{h}$, and $\mathfrak{b}$ the Lie algebras of $G,T$, and $B$, respectively. Let $\Phi$ be the root system given by the pair $(\mathfrak{g},\mathfrak{h})$, and let $\Phi^+$ denote the set of positive roots corresponding to $\mathfrak{b}$. Throughout this paper, we set $d:=|\Phi^+|$.\\
\indent Let $P_+(\mathfrak{g})$ denote the set of dominant integral weights. To each weight $\lambda\in P_+(\mathfrak{g})$, let $L(\lambda)$ denote the irreducible representation of $G$ with highest weight $\lambda$, and denote by $(,)$ the non-degenerate bilinear form on $\mathfrak{h}^*$ induced by the Killing form. Then the following is well known (see, for example, p.336 in \cite{MR2522486}).
\begin{WDF}
Let $\lambda\in P_+(\mathfrak{g})$. Then
\begin{center}
$dim(L(\lambda))=\displaystyle\prod_{\alpha\in\Phi^+}\frac{(\lambda+\rho,\alpha)}{(\rho,\alpha)}$,
\end{center}
where $\rho$ denotes $\displaystyle\frac{1}{2}\sum_{\alpha\in\Phi^+}\alpha$.
\end{WDF}
\noindent Following the notation in \cite{MR2906911}, let $c_{\lambda}(\alpha):=\displaystyle\frac{(\lambda,\alpha)}{(\rho,\alpha)}$. Then the above formula can be written as $dim(L(\lambda))=\displaystyle\prod_{\alpha\in\Phi^+}(c_{\lambda}(\alpha)+1)$.\\
\indent Given a graded $\mathbb{C}$-algebra $A$ with $i$th homogeneous component $A_i$, we define its \emph{Hilbert function} to be the map $HF_A:\mathbb{N}\rightarrow\mathbb{N}$, given by $HF_A(i)=dim(A_i)$. Then the \emph{Hilbert series} of $A$ is the formal power series
\begin{center}
$HS_q(A):=\displaystyle\sum_{n\in\mathbb{N}}HF_A(n)q^n$.
\end{center}
We give some basic properties of the Hilbert function and series for a graded $\mathbb{C}$-algebra $A$. For further information, see, for example, \cite{MR1669884},\cite{MR1416564}. If $A$ is generated by $A_1$, then the Hilbert series of $A$ must represent a rational function of the form
\begin{center}
$\displaystyle\frac{p(q)}{(1-q)^d}$,
\end{center}
where $p(q)\in\mathbb{Z}[q]$ is a polynomial in $q$ with integer coefficients. Further, if we consider the variety given by the spectrum of $A$, then the dimension of this variety is $d$.
\\
\indent As a generalization of the above, if we have an $\mathbb{N}^k$-graded $\mathbb{C}$-algebra $A$ with homogeneous component $A_{(a_1,\dots,a_k)}$ corresponding to the element $(a_1,\dots,a_k)\in\mathbb{N}^k$, we can define its \emph{$\mathbb{N}^k$-graded Hilbert series} as the formal powers series
\begin{center}
$\displaystyle\sum_{(a_1,\dots,a_k)\in\mathbb{N}^k}dim(A_{(a_1,\dots,a_k)})q_1^{a_1}\dots q_k^{a_k}$.
\end{center}
This series can be restricted via a substitution to a single grading on $A$. For example, we could make the substitution $q_i\mapsto q$ to get a Hilbert series for $A$. Note that different restrictions correspond to different gradations of $A$, and these  may give different Hilbert series.
\\
\indent In \cite{MR2906911}, the authors are interested in computing the Hilbert series of the homogeneous coordinate ring of an equivariant embedding of $G/P$ into a projective space. We recall the details. Given any irreducible highest weight representation $L(\lambda)$ of $G$, we can consider the parabolic subgroup given by the stabilizer of the unique hyperplane $H$ in $L(\lambda)$ fixed by the Borel subgroup $B$. If we denote by $\mathbb{P}(L(\lambda))$ the projective space of all hyperplanes in $L(\lambda)$, then we have an embedding
\begin{center}
$\pi_{\lambda}:G/P\rightarrow\mathbb{P}(L(\lambda))$,
\end{center}
given by the formula $\pi_{\lambda}(gP):=g(H)$. Then it is a consequence of the Borel-Weil theorem that the homogeneous coordinate ring $A_{\lambda}(G/P)$ is a sum of highest weight representations. Namely,
\begin{center}
$A_{\lambda}(G/P)=\displaystyle\bigoplus_{n\in\mathbb{N}}L(n\lambda)$.
\end{center}
Thus, the Hilbert series of the embedding is $\displaystyle\sum_{n\in\mathbb{N}}dim(L(n\lambda))q^n$. The authors then prove that this series has the following closed form.
\begin{thmn}
The Hilbert series of the embedding $\pi_{\lambda}$ of $G/P$ is
\begin{center}
$\displaystyle\prod_{\alpha\in\Phi^+}\left(c_{\lambda}(\alpha)q\frac{d}{dq}+1\right)\frac{1}{1-q}$.
\end{center}
\end{thmn}
We want to extend this to a multi-variate series graded over finitely many dominant integral weights. To this end, we use the notation $\textbf{a}$ for a $k$-tuple $(a_1,\dots,a_k)\in\mathbb{N}^k$. We use the convention that $\textbf{a}^{\textbf{i}}:=a_1^{i_1}\dots a_k^{i_k}$ for two $k$-tuples $\textbf{a}$ and $\textbf{i}$. We denote by $|\textbf{i}|$ the sum of the indices $i_1+\dots+i_k$. Given partial derivatives $\frac{\partial}{\partial q_i}$ and a $k$-tuple $\textbf{i}$, we use the notation $\displaystyle\left(\frac{\partial}{\partial\textbf{q}}\right)^{\textbf{i}}$ for the product $\displaystyle\left(\frac{\partial}{\partial q_1}\right)^{i_1}\dots\left(\frac{\partial}{\partial q_k}\right)^{i_k}$.
\\
\indent We use the notation $\langle\lambda_1,\dots,\lambda_k\rangle$ to denote the lattice cone in the dominant chamber generated by the dominant integral weights $\lambda_1,\dots,\lambda_k$, and consider the following series.
\begin{center}
$HS_{\textbf{q}}\langle\lambda_1,\dots,\lambda_k\rangle:=\displaystyle\sum_{\textbf{a}\in\mathbb{N}^k}dim(L(a_1\lambda_1+\dots+a_k \lambda_k))\textbf{q}^{\textbf{a}}$
\end{center}
We prove the following.
\begin{mthm}
Let $\lambda_1,\dots,\lambda_k$ be dominant integral weights. Then
\begin{equation}
HS_{\textbf{q}}\langle\lambda_1,\dots,\lambda_k\rangle=\displaystyle\prod_{\alpha\in\Phi^+}\left(1+c_{\lambda_1}(\alpha)q_1\frac{\partial}{\partial q_1}+\dots+c_{\lambda_k}(\alpha)q_k\frac{\partial}{\partial q_k}\right)\prod_{i=1}^k\frac{1}{1-q_i},
\end{equation}
where $c_{\lambda}(\alpha):=\displaystyle\frac{(\lambda,\alpha)}{(\rho,\alpha)}$
\end{mthm}
We will use the above formula to compute the Hilbert series for certain determinantal varieties. To this end, we define the \emph{determinantal variety of rank k} to be the subset of all rank at most $k$ matrices in $M_{m,n}(\mathbb{C})$. The \emph{symmetric determinantal variety of rank $k$} is the subset of rank at most $k$ matrices in $Sym_n(\mathbb{C}):=\{X\in M_n(\mathbb{C})\mid X-X^T=0\}$. The \emph{anti-symmetric determinantal variety of rank 2$k$} is the subset of rank at most 2$k$ matrices in $ASym_n(\mathbb{C}:=\{X\in M_{2n}(\mathbb{C})\mid X+X^T=0\}$. We denote these three varieties as $\mathcal{D}^{\leq k}_{m,n}, \mathcal{SD}^{\leq k}_n$, and $\mathcal{AD}^{\leq 2k}_n$, respectively.

\section{Proof of the main theorem}
\begin{thm}
Let $\lambda_1,\dots,\lambda_k$ be dominant integral weights. Then
\begin{center}
$HS_{\textbf{q}}\langle\lambda_1,\dots,\lambda_k\rangle=\displaystyle\prod_{\alpha\in\Phi^+}\left(1+c_{\lambda_1}(\alpha)q_1\frac{\partial}{\partial q_1}+\dots+c_{\lambda_k}(\alpha)q_k\frac{\partial}{\partial q_k}\right)\prod_{i=1}^k\frac{1}{1-q_i}$,
\end{center}
where $c_{\lambda}(\alpha):=\displaystyle\frac{(\lambda,\alpha)}{(\rho,\alpha)}$.
\end{thm}

\begin{proof}
By the Weyl Dimension Formula, we have
\begin{equation}
HS_{\textbf{q}}\langle\lambda_1,\dots,\lambda_k\rangle=\displaystyle\sum_{\textbf{a}\in\mathbb{N}^k}\prod_{\alpha\in\Phi^+}(1+a_1c_{\lambda_1}(\alpha)+\dots+a_kc_{\lambda_k}(\alpha))\textbf{q}^{\textbf{a}}.
\end{equation},
where $\textbf{a}:=(a_1,\dots,a_k)$, and $\textbf{q}^{\textbf{a}}:=q_1^{a_1}\dots q_k^{a_k}$. Consider the product
\begin{center}
$\displaystyle\prod_{\alpha\in\Phi^+}(1+a_1c_{\lambda_1}(\alpha)+\dots+a_kc_{\lambda_k}(\alpha))$.
\end{center}
It is a polynomial in the $a_i$ for $1\leq i\leq k$. So we have
\begin{equation}
\displaystyle\prod_{\alpha\in\Phi^+}(1+a_1c_{\lambda_1}(\alpha)+\dots+a_kc_{\lambda_k}(\alpha))=\sum_{|\textbf{i}|\leq d} b_{\textbf{i}}\textbf{a}^{\textbf{i}},
\end{equation}
where $d:=|\Phi^+|$, and $|\textbf{i}|:=i_1+\dots+i_k$. The coefficients do not depend on $\textbf{a}$. Thus, (3.1) becomes
\begin{equation}
\displaystyle\sum_{|\textbf{i}|\leq d}b_{\textbf{i}}\sum_{\textbf{a}\in\mathbb{N}^k}\textbf{a}^{\textbf{i}}\textbf{q}^{\textbf{a}}.
\end{equation}
We now find a closed form for $\displaystyle\sum_{\textbf{a}\in\mathbb{N}^k}\textbf{a}^{\textbf{i}}\textbf{q}^{\textbf{a}}$. Note that if we define 
\begin{center}
$f_{(i_1,\dots,i_k)}(\textbf{q}):=\displaystyle\sum_{\textbf{a}\in\mathbb{N}^k}\textbf{a}^{\textbf{i}}\textbf{q}^{\textbf{a}}$, 
\end{center}
then hitting $f_{(i_1,\dots,i_k)}(\textbf{q})$ with the partial differential operator $q_j\displaystyle\frac{\partial}{\partial q_j}$ gives us 
\\
$f_{(i_1,\dots,i_j+1,\dots,i_k)}(\textbf{q})$. Since $f_{(0,\dots,0)}(\textbf{q})=\displaystyle\prod_{j=1}^k\frac{1}{1-q_k}$, and these operators commute, we have
\begin{center}
$f_{(i_1,\dots,i_k)}(\textbf{q})=\displaystyle\left(q_1\frac{\partial}{\partial q_1}\right)^{i_1}\dots\left(q_k\frac{\partial}{\partial q_k}\right)^{i_k}\prod_{j=1}^k\frac{1}{1-q_j}=\left(\frac{\partial}{\partial\textbf{q}}\right)^{\textbf{i}}\prod_{j=1}^k\frac{1}{1-q_j}$.
\end{center}
Therefore, (3.3) becomes
\begin{equation}
\displaystyle\sum_{|\textbf{i}|\leq d}b_{\textbf{i}}\left(\frac{\partial}{\partial\textbf{q}}\right)^{\textbf{i}}\prod_{j=1}^k\frac{1}{1-q_j}.
\end{equation}
Then, we have
\begin{center}
$\displaystyle\sum_{|\textbf{i}|\leq d}b_{\textbf{i}}\left(\frac{\partial}{\partial\textbf{q}}\right)^{\textbf{i}}=\prod_{\alpha\in\Phi^+}\left(1+c_{\lambda_1}(\alpha)\frac{\partial}{\partial q_1}+\dots+c_{\lambda_k}(\alpha)\frac{\partial}{\partial q_k}\right)$,
\end{center}
since this is just (3.2) with the substitution $a_i\mapsto\displaystyle\frac{\partial}{\partial q_i}$. The result follows.
\end{proof}

\section{Examples}

By setting $k=1$, we obtain the Hilbert series of an equivariant embedding of a projective variety, as in the case of \cite{MR2906911}. If $G$ has rank $k$, and we look at the formal power series given by the fundamental dominant weights $\langle\omega_1,\dots,\omega_k\rangle$, we obtain a generating function for the dimensions of the irreducible representations of $G$, as in \cite{MR1120029}.
\\
\indent Inside our given Borel subgroup $B\subset G$, we have a maximal unipotent subgroup $U$ such that $B=TU$. The quotient $G/U$ has a natural structure as an affine variety. This variety has coordinate ring
\begin{equation}
\mathbb{C}[G/U]\cong\displaystyle\bigoplus_{\lambda\in P_+(\mathfrak{g})}\mathbb{C}_{\lambda}\otimes L(\lambda),
\end{equation}
(see, for example, \S 3.3 in \cite{MR2401818}). If we set $V(\lambda):=\mathbb{C}_{\lambda}\otimes L(\lambda)$, we have a gradation on $\mathbb{C}[G/U]$ given by $V(\lambda)V(\mu)=V(\lambda+\mu)$. Then replacing $P_+(\mathfrak{g})$ with a lattice cone $\langle\lambda_1,\dots,\lambda_k\rangle$ gives a subalgebra of $\mathbb{C}[G/U]$. The spectrum of this subalgebra is a variety, and this variety has an $\mathbb{N}^k$-graded Hilbert series given by $HS_{\textbf{q}}\langle\lambda_1,\dots,\lambda_k\rangle$. 
\\
\indent Our main interest in examples is going to be using the formula from the main theorem to find a series in $k$ variables and then specializing that series to a Hilbert series on the underlying variety given by subalgebras of (4.1) corresponding to a given lattice cone $\langle\lambda_1,\dots,\lambda_k\rangle$.
\\
\indent Some interesting examples are those given by looking at the homogeneous coordinate ring of the three determinantal varieties $\mathcal{D}^{\leq k}_{m,n},\mathcal{SD}^{\leq k}_n$, and $\mathcal{AD}^{\leq 2k}_n$. We begin with the symmetric determinantal varieties $\mathcal{SD}^{\leq k}_n$. Note that finding the Hilbert series for these varieties is in general a very difficult thing to do (see, for example, \cite{MR2037715}, \cite{enright}).\\
\indent The Second Fundamental Theorem of Invariant Theory for $O(n)$ (see, for example, \cite{MR2522486}, p.561), states that the homogeneous coordinate ring $\mathcal{SD}^{\leq k}_n$ decomposes as an $SL(n,\mathbb{C})$-module in the following way:
\begin{center}
$\mathbb{C}[\mathcal{SD}^{\leq k}_n]\cong\displaystyle\bigoplus_{\lambda}L(\lambda)$,
\end{center}
where $\lambda$ runs over all even dominant integral weights of depth at most $k$. Here, an even weight of depth at most $k$ is one that lies in the lattice cone $\langle2\omega_1,\dots,2\omega_k\rangle$, where $\omega_1,\dots,\omega_{n-1}$ are the fundamental dominant weights of $SL(n,\mathbb{C})$, and we are using the standard Borel subgroup of upper triangular matrices in $SL(n,\mathbb{C})$.\\
\indent So we can compute the series $HS_{\textbf{q}}\langle2\omega_1,\dots,2\omega_k\rangle$ and specialize the variables in an appropriate way to recover the Hilbert series of the standard embedding of the symmetric determinantal variety. The standard Hilbert series on $\mathcal{SD}^{\leq k}_n$ is given by
\begin{center}
$\displaystyle\sum_{\lambda}dim(L(\lambda))q^{|\lambda|}$,
\end{center}
where again, $\lambda$ runs over all even dominant integral weights of depth at most $k$. After computing the series $HS_{\textbf{q}}\langle2\omega_1,\dots,2\omega_k\rangle$, we specialize to the standard Hilbert series by making the substitution $q_i\mapsto q^i$ for $i=1,\dots,k$.
\\
\indent We now compute some examples. We consider the variety $\mathcal{SD}^{\leq 2}_4$. Then we compute the series $HS_{\textbf{q}}\langle2\omega_1,2\omega_2\rangle$, where $\omega_1$ and $\omega_2$ are the first two fundamental dominant weights of $SL(4,\mathbb{C})$. The main theorem gives us the following closed form for $HS_{\textbf{q}}\langle2\omega_1,2\omega_2\rangle$:
\begin{center}
$\displaystyle\prod_{1\leq i<k\leq 4}\left(1+2c_{\omega_1}(\epsilon_i-\epsilon_j)q_1\frac{\partial}{\partial q_2}+2c_{\omega_2}(\epsilon_i-\epsilon_j)q_2\frac{\partial}{\partial q_2}\right)\frac{1}{(1-q_1)(1-q_2)}$,
\end{center}
where $\Phi^+=\{\epsilon_i-\epsilon_j\mid 1\leq i<j\leq 4\}$, and $\epsilon_i$ is the functional that gives the $i$th diagonal element of a matrix in $\mathfrak{g}=\mathfrak{sl}(4,\mathbb{C})$. Then computing $c_{\omega_1}(\epsilon_i-\epsilon_j)$ and $c_{\omega_2}(\epsilon_i-\epsilon_j)$ for $1\leq i<j\leq 4$ gives us
\begin{center}
$(1+2q_1\frac{\partial}{\partial q_1})(1+2q_2\frac{\partial}{\partial q_2})(1+q_1\frac{\partial}{\partial q_1}+q_2\frac{\partial}{\partial q_2})(1+q_2\frac{\partial}{\partial q_2})(1+\frac{2}{3}q_1\frac{\partial}{\partial q_1}+\frac{2}{3}q_2\frac{\partial}{\partial q_2})\frac{1}{(1-q_1)(1-q_2)}$.
\end{center}
Applying the differential operators then yields
\begin{center}
$\frac{1+6q_1+15q_2+q_1^2+16q_1q_2+15q_2^2+q_2^3-50q_1q_2^2-29q_1^2q_2-4q_1q_2^3 -25q_1^2q_2^2+6q_1^3q_2+21q_1^2q_2^3+20q_1^3q_2^2+6q_1^3q_2^3}{(1-q_1)^4(1-q_2)^5}$.
\end{center}
This formula seems unmanagable, but is easy to compute with Mathematica or Maple, and after we make the substitution $q_i\mapsto q^i$, we get
\begin{center}
$\displaystyle\frac{1+3q+6q^2}{(1-q)^7}$,
\end{center}
which is the Hilbert series for the standard embedding of $\mathcal{SD}^{\leq 2}_4$.
\\
\indent We can then increase the size of the matrices to get a recursive way of finding the Hilbert series of $\mathcal{SD}^{\leq 2}_n$. Let $\{\alpha_1,\dots,\alpha_{n-1}\}$ be the simple roots of $SL(n,\mathbb{C})$. The only positive roots of $SL(n,\mathbb{C})$ that contribute to the product in $HS_{\textbf{q}}\langle\omega_1,\omega_2\rangle$, are those which can be written as a string of simple roots $\sum\alpha_i$ beginning at either $\alpha_1$ or $\alpha_2$. So, as we go from $n-1$ to $n$, we add two differential operators, namely, those which correspond to the positive roots $\alpha_2+\dots+\alpha_{n-1}$ and $\alpha_1+\dots+\alpha_{n-1}$. If we define $HS_{\textbf{q}}^n\langle\omega_1,\omega_2\rangle$ to be the series given by the first two fundamental dominant weights of $SL(n,\mathbb{C})$, we have the following recursive formula.

\begin{lem}
For $n\geq 3$,
\begin{center}
$HS_{\textbf{q}}^n\langle2\omega_1,2\omega_2\rangle=(1+\frac{2}{n-2}q_2\frac{\partial}{\partial q_2})(1+\frac{2}{n-1}q_1\frac{\partial}{\partial q_1}+\frac{2}{n-1}q_2\frac{\partial}{\partial q_2})HS_{\textbf{q}}^{n-1}\langle2\omega_1,2\omega_2\rangle$.
\end{center}
\end{lem}

We obtain the recursion by simply computing the coefficients for the two new weights. Note that this is a linear recursion on the multi-variate series, but it does not pass to a recursion on the single variable Hilbert series for the varieties $\mathcal{SD}^{\leq 2}_n$. In this way, the multi-variate series behaves more nicely than the single variable Hilbert series. This multi-variate series then allows us to more easily compute the Hilbert series of $\mathcal{SD}^{\leq k}_n$.\\
\indent We have a similar story for the Hilbert series of the standard embedding of $\mathcal{AD}^{\leq 2k}_n$. The Second Fundamental Theorem of Invariant Theory for $Sp(2n,\mathbb{C})$ (see, for example, p. 562 in \cite{MR2522486}) says that the homoegenous coordinate ring of $\mathcal{AD}^{\leq 2k}_n$ decomposes as an $SL(2n,\mathbb{C})$-module as
\begin{center}
$\mathbb{C}[\mathcal{AD}^{\leq 2k}_n]\cong\displaystyle\sum_{\lambda}L(\lambda)$,
\end{center}
where $\lambda$ runs over the lattice cone $\langle\omega_2,\omega_4,\dots,\omega_{2k}\rangle$. Then $HS_{\textbf{q}}\langle\omega_2,\omega_4,\dots,\omega_{2k}\rangle$ can again be specialized to the standard Hilbert series given by
\begin{center}
$\displaystyle\sum_{\lambda}dim(L(\lambda))q^{|\lambda|}$,
\end{center}
where $\lambda$ runs over $\langle\omega_2,\omega_4,\dots,\omega_{2k}\rangle$, by making the substitution $q_i\mapsto q^i$.
\\
\indent We finish with an example beyond the scope of the determinantal varieties. We consider the lattice cone $\langle3\omega_1,3\omega_2\rangle$ in the weight lattice $P_+(\mathfrak{sl}(3,\mathbb{C}))$. This series corresponds to the subalgebra
\begin{center}
$\displaystyle\bigoplus_{\lambda\in\langle3\omega_1,3\omega_2\rangle}\mathbb{C}_{\lambda}\otimes L(\lambda)$
\end{center}
of (4.1). If we take the torus $T\cong (\mathbb{C}^{\times})^2$ of $SL(3,\mathbb{C})$, then this algebra corresponds to a coordinate of the variety $G/AU$, where $A$ is the finite subgroup of $T$ generated by $\{(\zeta_3,1),(1,\zeta_3)\}$, where $\zeta_3$ is a primitive third root of unity.
\\
\indent From (1.2), we have the following closed form for $HS_{\textbf{q}}\langle3\omega_1,3\omega_2\rangle$:
\begin{center}
$\displaystyle\prod_{1\leq i<j\leq 3}\left(1+3c_{\omega_1}(\epsilon_i-\epsilon_j)q_1\frac{\partial}{\partial q_1}+3c_{\omega_2}(\epsilon_i-\epsilon_j)q_2\frac{\partial}{\partial q_2}\right)\frac{1}{(1-q_1)(1-q_2)}$,
\end{center}
where $\Phi^+=\{\epsilon_i-\epsilon_j\mid 1\leq i<j\leq 3\}$. Then computing the coefficients yields
\begin{center}
$\displaystyle(1+3q_1\frac{\partial}{\partial q_1})(1+3q_2\frac{\partial}{\partial q_2})(1+\frac{3}{2}q_1\frac{\partial}{\partial q_1}+\frac{3}{2}q_2\frac{\partial}{\partial q_2})\frac{1}{(1-q_1)(1-q_2)}$,
\end{center}
and computing the partial derivatives yields
\begin{center}
$\displaystyle\frac{1+7q_1+7q_2+q_1^2+q_2^2+13q_1q_2-11q_1^2q_2-11q_1q_2^2+8q_1^2q_2^2}{(1-q_1)^3(1-q_2)^3}$.
\end{center}
We can perform various substitutions to find nice formulas for Hilbert series of different embeddings of $G/AU$. For instance, after performing the substitution $q_i\mapsto q$, we get
\begin{center}
$\displaystyle\frac{1+14q+15q^2-22q^3+8q^4}{(1-q)^6}$.
\end{center}
We reduce the complexity in performing the calculation of the above Hilbert series by first computing the series in two variables, and then specializing it to a Hilbert series for $G/AU$. \\
\indent The series $HS_{\textbf{q}}\langle\lambda_1,\dots,\lambda_k\rangle$ is useful in computing these series, as it bypasses much of the complicated machinery normally used in their computation. The multi-variate series applies to any variety whose coordinate ring can be decomposed into highest weight representations whose weights run over a lattice cone in $P_+(\mathfrak{g})$. These include the three classes of determinantal varieties, but also include many other interesting examples, such as the flag variety $G/U$.

\end{document}